\documentclass[11pt,english]{article}
\usepackage[left=1.5cm, right=1.5cm, top=2cm, bottom=2cm]{geometry}

\usepackage[toc]{appendix}
\usepackage[colorlinks=true,linkcolor=black,citecolor=black,filecolor=black,linkcolor=black,urlcolor=black]{hyperref}
\usepackage{mathrsfs}
\usepackage{relsize}
\usepackage[T1]{fontenc}
\usepackage{lmodern}
\usepackage[utf8]{inputenc}
\usepackage{indentfirst}
\usepackage{nomencl}
\usepackage{color}
\usepackage{graphicx}
\usepackage{microtype}
\usepackage{lipsum}
\usepackage{multicol}										
\usepackage{mathtools}
\usepackage{latexsym}				
\usepackage{amssymb,amsmath,amsthm,amsfonts,amscd,amstext,amsbsy}
\usepackage{enumerate}
\usepackage{todonotes}
\usepackage{enumitem}

\DeclareMathOperator{\conex}{\ds\mathlarger{\mathlarger{\nabla}}}
\DeclareMathOperator{\tr}{tr}

\DeclareMathOperator{\dd}{d}

\newcommand{\Q}{\mathfrak{Q}}
\newcommand{\Hc}{\mathcal{H}}

\DeclareMathOperator{\Hess}{\nabla^2}

\DeclareMathOperator{\dist}{dist}

\DeclareMathOperator{\diam}{diam}
\DeclareMathOperator{\R}{\mathbb{R}}
\DeclareMathOperator{\HH}{\mathbb{H}}
\DeclareMathOperator{\Ricc}{Ricc}
\newcommand{\W}{\Omega}
\newcommand{\norm}[1]{\left\| #1 \right\|}
\newcommand{\modulo}[1]{\left| #1 \right|}
\newcommand{\cyl}{\mathcal{C}}
\newcommand{\ocyl}{\tilde{\mathcal{C}}}
\usepackage{mathrsfs}
\newcommand{\cl}{\mathscr{C}}
\newcommand{\ds}{\displaystyle}

\newcommand{\Eum}{\ds\mathsmaller{\frac{\partial}{\partial x_1}}}

\newcommand{\Ei}{\ds\mathsmaller{\frac{\partial}{\partial x_i}}}
\newcommand{\Ej}{\ds\mathsmaller{\frac{\partial}{\partial x_j}}}
\newcommand{\Ek}{\ds\mathsmaller{\frac{\partial}{\partial x_k}}}

\newcommand{\Dz}{\partial_z}

\newcommand{\Di}{\partial_i}

\newcommand{\Dk}{\partial_k}

\newcommand{\Dij}{\partial_{ij}}
\newcommand{\Dii}{\partial_{ii}}
\newcommand{\Dkk}{\partial_{kk}}

\newcommand{\Dkkum}{\partial_{kk1}}

\newcommand{\Dum}{\partial_1}
\newcommand{\Dumk}{\partial_{1k}}

\newcommand{\Dumi}{\partial_{1i}}
\newcommand{\Dumum}{\partial_{11}}
\newcommand{\Dumumum}{\partial_{111}}
\newcommand{\Dumij}{\partial_{1ij}}
\newcommand{\Dumii}{\partial_{1ii}}

\DeclareMathOperator{\Gr}{Gr}

\newcommand{\cYY}{\overline{\conex}_Y Y}
\newcommand{\nH}{\overline{\nabla} H}

\newcommand{\oM}{\overline{M}}
\newcommand{\oconex}{\overline{\conex}}
\newcommand{\oRicc}{\overline{\Ricc}}

\newcommand{\escalar}[2]{{\left\langle #1,#2 \right\rangle}}

\makeatletter
\newcommand{\pushright}[1]{\ifmeasuring@#1\else\omit\hfill$\displaystyle#1$\fi\ignorespaces}
\newcommand{\pushleft}[1]{\ifmeasuring@#1\else\omit$\displaystyle#1$\hfill\fi\ignorespaces}
\makeatother

\newtheorem{teo}{Theorem}
\newtheorem{prop}[teo]{Proposition}

\newtheorem{cor}[teo]{Corollary}
\newtheorem*{teo*}{Theorem}

\theoremstyle{definition}

\newtheorem{obs}[teo]{Remark}

\usepackage{amsthm}
\makeatletter
\def\th@plain{%
  \thm@notefont{}
  \itshape 
}
\def\th@definition{%
  \thm@notefont{}
  \normalfont 
}
\makeatother

\makeatletter
\let\oldr@@t\r@@t
\def\r@@t#1#2{%
\setbox0=\hbox{$\:\oldr@@t#1{#2\,}$}\dimen0=\ht0
\advance\dimen0-0.2\ht0
\setbox2=\hbox{\vrule height\ht0 depth -\dimen0}%
{\box0\lower0.4pt\box2}}
\LetLtxMacro{\oldsqrt}{\sqrt}
\renewcommand*{\sqrt}[2][]{\oldsqrt[#1]{#2}}
\makeatother

\newcommand\blfootnote[1]{%
  \begingroup
  \renewcommand\thefootnote{}\footnote{#1}%
  \addtocounter{footnote}{-1}%
  \endgroup
}

\begin{document}
\title{Prescribing the curvature to Killing Graphs}
\author{Yunelsy N. Alvarez}
\date{}
\maketitle

\blfootnote{The author was financed by the Coordenação de Aperfeiçoamento de Pessoal de Nível Superior - Brasil (CAPES) - Finance Code 001.}
\blfootnote{\ ~ \\
Yunelsy N. Alvarez\\
Departamento de Matemática, Instituto de Matemática e Estadística, Universidade de São Paulo\\
São Paulo, Brazil,  CEP 05508-090 \\
Email addresses: ynapolez@gmail.com; ynalvarez@usp.br
}

\begin{abstract}
In this work we prove the existence and uniqueness of Killing graphs with prescribed mean
curvature considering functions which are not necessarily constant along the flow lines of the Killing vector field.
\end{abstract}


\section{Introduction}
\label{secIntroducao}


Let $\overline{M}$ be a complete Riemannian manifold of dimension $n+1$, $n\geq 2$. Let us suppose that $\overline{M}$ is endowed with a non-singular Killing vector field $Y$ whose orthogonal distribution is integrable. We fix some integral leaf $M$ and we denote by $\Phi:M\times\R\rightarrow\overline{M}$ the flow generated by $Y$. The Killing graph of a function $u$ defined over a bounded domain $\W\subset M$, firstly defined by Dajczer-Hinojosa-Lira in \cite{Dajczer2008}, is the hypersurface
\begin{equation}
\Gr u =\{\Phi(x,u(x));\ x\in\W\}.
\label{eq:KillingGraph}
\end{equation}

If $H(\Phi(x,u(x)))$ is the mean curvature of $\Sigma$ at each point $\Phi(x,u(x))$, then $u$ necessarily satisfies the equation 
\begin{equation}\label{operador_minimo_1}
W^2\Delta u -\Hess u(\nabla u,\nabla u)-\rho^{-2}\left(\rho^{-2}+W^2\right)\escalar{\cYY}{\nabla u} = nH(\Phi(x,u))W^3,
\end{equation}
where 
\begin{equation}
\rho=\rho(x)=\norm{Y}_{\overline{M}},
\label{eq:rho}
\end{equation}
\begin{equation}
W=\sqrt{\rho^{-2}+\norm{\nabla u}^2}.
\label{eq:theW}
\end{equation} 

The main goal in this work is to study the Dirichlet problem for equation \eqref{operador_minimo_1}. Before to state the principal theorem we need to fix the notation. 

\bigskip
\noindent\textbf{Notation.} 
\begin{enumerate}
\item {Geometric entities on $\oM$ will be marked with a bar and any geometric objects without this superscript will be assumed to lie on $M$}.

\item For an hypersurface $\Gamma$ in $M$, we called the hypersurface 
$$ \overline{\Gamma}=\{\Phi(y,z),\ y\in\Gamma, \ z\in\R\}\subset \overline{M}$$
the translation of $\Gamma$ along the flow of $Y$. 
In particular, for a bounded domain $\W\subset M$
\begin{equation}
\cyl =\{\Phi(y,z);\ y\in\partial\W,\ z\in\R\}
\label{eq:cyl}
\end{equation}
is the cylinder obtained by ``translating'' $\partial\W$ along the flow lines, 
\begin{equation}
\ocyl =\{\Phi(x,z);\ x\in\overline{\W},\ z\in\R\}
\label{eq:cyl_solid}
\end{equation}
is the ``solid'' cylinder bounded by $\cyl$ 
and, for $\delta>0$,
\begin{equation}
\ocyl({\delta}) =\{\Phi(x,z);\ x\in\overline{\W},\ -{\delta}\leq z\leq {\delta}\}
\label{eq:cyl_solid_part}
\end{equation}
is the solid cylinder ``between'' the hypersurfaces $\Phi_{-\delta}(M)$ and $\Phi_{\delta}(M)$. 

\item {The set of points in $\W$ having a unique nearest point to $\partial\W$ will be denoted by $\W_0$. Besides, for each $x\in \W_0$, $y(x)$ is the point in $\partial\W$ realizing the distance $\dd(x):=\dist (x,\partial\W)$. Furthermore, $\gamma_y(t)$, $0\leq t < \tau(y)$, is the inner normal geodesic to the boundary, being that $\gamma_y(\tau(y))$ is the point of $\gamma_y$ in $\partial\W_0\setminus \partial\W$. }

\item For every hypersurface $\Gamma$ in $M$ we will denote by $\Hc_{\Gamma}$ the mean curvature of $\Gamma$ in $M$, and $\Hc_{\overline{\Gamma}}$ will denote the mean curvature of its translation $\overline{\Gamma}$ in $\overline{M}$. 

\item Let  $\rho_i=\inf\limits_{\W}\rho$ and $\rho_s=\sup\limits_{\W}\rho$.

\item Over $\cl^2(\overline{\W})$ we define the operators
\begin{equation}\label{operador_Q}
\Q u:=W^2\Delta u -\Hess u(\nabla u,\nabla u)-\rho^{-2}\left(\rho^{-2}+W^2\right)\escalar{\cYY}{\nabla u} - nH(\Phi(x,u))W^3.
\end{equation}
and
\begin{equation}\label{operador_Qsigma}
\Q_{\sigma} u:=W^2\Delta u -\Hess u(\nabla u,\nabla u)-\rho^{-2}\left(\rho^{-2}+W^2\right)\escalar{\cYY}{\nabla u} - n\sigma H(\Phi(x,u))W^3.
\end{equation}
where $\sigma \in[0,1]$. 
\end{enumerate}
\bigskip

Having set the notation, we state the principal results of this work:
\begin{teo}\label{T_exist_Ricci}
Let $\Omega \subset M$ be a $\cl^{2,\alpha}$ bounded domain for some $\alpha\in(0,1)$ and $\varphi\in\cl^{2,\alpha}(\overline{\W})$. 
Let $H\in\cl^{1,\alpha}(\ocyl)$ satisfying 
\begin{equation}
\escalar{\nH}{Y}_{\oM}\geq 0 \mbox{ in } \ocyl
\label{eq:HY_intro}
\end{equation}
and
\begin{equation}\label{cond_H_Ricci_exist}
\norm{\nH\left(\Phi(x,\varphi(y(x)))\right)}_{\oM}\leq \left(H\left(\Phi(x,\varphi(y(x)))\right)\right)^2 + \frac{\oRicc_x}{n} \ \ \forall \ x\in\W_0.
\end{equation}
If
\begin{equation}\label{StrongSerrinCondition_mainThm}
\Hc_{\cyl}(y)\geq \sup_{\sigma\in[0,1]}\modulo{H(\Phi(y,\sigma\varphi(y)))} \ \forall \ y\in\partial\W, 
\end{equation}
then there exists a unique function $u\in\cl^{2,\alpha}(\overline{\W})$ satisfying $u|_{\partial\W}=\varphi$ and whose Killing graph has mean curvature $H$. 
\end{teo}

%
\begin{cor}\label{T_exist_Ricci_cor}
Let $\Omega \subset M$ be a $\cl^{2,\alpha}$ bounded domain for some $\alpha\in(0,1)$. 
Let $H\in\cl^{1,\alpha}(\ocyl)$ satisfying 
\begin{equation}
\escalar{\nH}{Y}_{\oM}\geq 0 \mbox{ in } \ocyl, 
\label{eq:HY_intro_cor}
\end{equation}
and
\begin{equation}\label{cond_H_Ricci_exist_cor}
\sup_{z\in\R}\norm{\nH\left(\Phi(x,z)\right)}_{\oM}\leq \inf_{z\in\R}\left(H\left(\Phi(x,z)\right)\right)^2 + \frac{\oRicc_x}{n} \ \ \forall \ x\in\W_0.
\end{equation}
If
\begin{equation}\label{StrongSerrinCondition_mainThm_cor}
\Hc_{\cyl}(y)\geq \sup\limits_{z\in\R}\modulo{H(\Phi(y,z))} \ \forall \ y\in\partial\W,
\end{equation}
then for every $\varphi\in\cl^{2,\alpha}(\overline{\W})$ there exists a unique function $u\in\cl^{2,\alpha}(\overline{\W})$ satisfying $u|_{\partial\W}=\varphi$ and whose Killing graph has mean curvature $H$. 
\end{cor}

Existence of graphs with prescribed mean curvature is a subject that has been studied for more than a hundred years. The pioneers were Berstein \cite{Bernstein}, Douglas \cite{Douglas1931}, Radó \cite{Rado1930} and Finn \cite{Finn1965} in the three-dimensional Euclidean space. Subsequently, Jenkins-Serrin \cite{Serrin1968} and Serrin \cite{Serrin} studied the problem in higher dimension. 

One conclusion we derive from those works is that there exists a sharp geometric condition to be satisfied by $\W$ in order to have a solution of the Dirichlet problem for the prescribed mean curvature equation for arbitrary boundary values (with further assumption concerning the regularity)\footnote{Sharp in the sense that if this condition fails we can construct a $\cl^{\infty}$ function over $\partial\W$ that can not be the boundary values of any solution of the corresponding mean curvature equation.}. 
For example, 
$\W$ must be mean convex to construct minimal graphs in $\R^{n+1}$ for each smooth boundary 
(see {\cite[Th. 1 p. 171]{Serrin1968}}). In the more general case, $H=H(x)$, Serrin proved that there exists a graph with mean curvature $H$ for arbitrarily given $\cl^2$ boundary values if and only if 
\begin{equation}\label{SerrinCondition}
(n-1)\Hc_{\partial\W}(y)\geq n\modulo{H(y)} \ \forall \ y\in\partial\W,
\end{equation}
provided the function $H\in\cl^1(\overline{\W})$ satisfies an additional (not geometric) hypothesis (see {\cite[Th. p. 484]{Serrin}}). 

{The present author and Sa Earp in \cite{artigonaoexist, alvarez2019existence} have generalized recently the aforementioned result of Serrin to the Riemannian product $M\times\R$.} 
We proved that for arbitrary smooth boundary data there exists a vertical graphs in $M\times\R$ with prescribed mean curvature $H=H(x,z)$ (satisfying some additional hypothesis) when 
\begin{equation}\label{StrongSerrinCondition}
(n-1)\Hc_{\partial\W}(y)\geq n \sup\limits_{z\in\R}\modulo{H\left(y,z\right)} \ \forall \ y\in\partial\W.
\end{equation}
Besides, we proved that \eqref{StrongSerrinCondition} is sharp also in this context. 


Previously, Dajczer-Hinojosa-Lira had already introduced the concept of Killing graphs. 
In this setting they proved (see {\cite[T. 1 p. 232]{Dajczer2008}}) that, 
for $H\in\cl^{\alpha}(\W)$ satisfying 
\begin{equation}\label{ricci_Djaczer}
\oRicc\geq - n \ds\inf_{\partial\W} \Hc_{\cyl}^2. 
\end{equation}
and $\varphi\in\cl^{2,\alpha}(\partial\W)$ given, 
there exists a unique function $u\in\cl^{2,\alpha}(\overline{\W})$ satisfying $u|_{\partial\W}=\varphi$ whose Killing graph has mean curvature $H$  if
\begin{equation}\label{serrin_Djaczer}
\ds\sup_{\W} \modulo{H} \leq \inf_{\partial\W} \Hc_{\cyl}.
\end{equation}

Theorem \ref{T_exist_Ricci} improves this result of Dajczer-Hinojosa-Lira in two senses. First of all because we are considering prescribed mean curvature functions that are not necessarily constant along the flow lines ($H=H(\Phi(x,z))$). Hypothesis \eqref{eq:HY_intro} guarantees that the maximum principle for operator $\Q$ holds. 

Secondly, because inequality \eqref{serrin_Djaczer} is weakened to a pointwise condition to be satisfied along the boundary. Indeed, if $H=H(x)$ inequality \eqref{StrongSerrinCondition_mainThm} becomes
\begin{equation}\label{StrongSerrinCondition_mainThm_Hx}
\Hc_{\cyl}(y)\geq \modulo{H(y)} \ \forall \ y\in\partial\W. 
\end{equation}
The key is to replace the additional geometric restriction \eqref{ricci_Djaczer} on $\W$ by an additional hypothesis on the function $H$ (see \eqref{cond_H_Ricci_exist}). 

We notice that, \eqref{StrongSerrinCondition_mainThm} seems to be the natural generalization of conditions \eqref{SerrinCondition} and \eqref{StrongSerrinCondition} because, analogously to the $\R^{n+1}$ and $M\times\R$ cases, it means that the mean curvature of the Killing cylinder must be greater than the mean curvature of the Killing graph of the function $u$ at each point of the intersection between them.  
The clues were given to us by two works 
whose focuses were existence of graphs-like with prescribed mean curvature in the half-space setting of the hyperbolic space $\HH^{n+1}$.

P.-A Nitsche \cite{Nitsche2002} studied radial graphs over bounded domains in the spherical cap of euclidean radius 1. 
He proved that there exists a graph over $\W$ with prescribed mean curvature a function $H(x)$ taking values on $[0,1]$ for arbitrary continuous boundary values if $\modulo{H(y)}<\Hc_{C}(y)$ everywhere on $\partial\W$; {while if $\modulo{H(y)}>\Hc_{C}(y)$ for some $y\in\partial\W$, it can be constructed a smooth curve that cannot be the boundary of any graph over $\W$ with curvature $H$.}

On the other hand, Guio-Sa Earp \cite{eliasarticle} considered a bounded domain $\W$ contained in a vertical hyperplane and studied the Dirichlet problem for the prescribed mean curvature equation for hypersurfaces which intersect at must in one point the horizontal horocycles orthogonal to $\W$. 
They obtained a sharp result when $\Hc_C(y) \geq \modulo{H(y)}$ for each $y\in\partial\W$ provided again that $\modulo{H(x)}\leq 1$ in $\W$.  

We highlight that the reason why Guio-Sa Earp and Nitsche have assumed that $H$ takes values on [-1,1] is to use the horoespheres as barriers in order to obtain a priori height estimates. Theorem \eqref{T_exist_Ricci} allows to consider functions taking values in the complement of this set. To see this note that, in this context, condition \eqref{cond_H_Ricci_exist} becomes
$$
\norm{\nH\left(\Phi(x,\varphi(y(x)))\right)}_{\oM}\leq \left(H\left(\Phi(x,\varphi(y(x)))\right)\right)^2 -1.
$$
\section{Transformation formulas}\label{transf_form}

In this section we compute some formulas that will help us throughout the text. 
First we establish a proposition that gives us a relation between the mean curvature of an hypersurfaces in $M$ and its translations along the flow lines of $Y$. 

\begin{prop}\label{Sup_transladadas}
Let $\Gamma$ be an embedded and oriented hypersurface in $M$ and $N$ a unit vector field normal to $\Gamma$. Let $\overline{\Gamma}$ be the translation of $\Gamma$ along the flow of $Y$. 
Then, for every $x\in\Gamma$, $N(x)$ is normal to $\overline{\Gamma}$ at $x$ and 
\begin{equation}
n\Hc_{\overline{\Gamma}}(x)=(n-1)\Hc_{\Gamma}(x)+\rho^{-2}\escalar{\oconex_Y Y}{N}, 
\label{eq:curv_oGamma}
\end{equation}
where $\Hc_{\overline{\Gamma}}$ and $\Hc_{\Gamma}$ are calculated with respect to $N$. 
\end{prop}
\begin{proof}
Recall that the translation of $\Gamma$ is given by 
$$\overline{\Gamma}=\left\{\Phi(x,z),\ x\in\Gamma, \ z\in\R\right\}.$$
Therefore, if $x_1,\dots, x_{n-1}$ are coordinates in $\Gamma$, then $x_1,\dots, x_{n-1},z$ are coordinates in $\overline{\Gamma}$. Consequently, $Y=\frac{\partial \Phi}{\partial z}$ is tangent to $\overline{\Gamma}$. For $x\in\Gamma$ we obviously have that $N(x)$ is orthogonal to $\overline{\Gamma}$ at $x$ because it is orthogonal to $Y$ (since $Y$ is normal to $M$) and to the remaining coordinate vectors $\Ei=\frac{\partial \Phi}{\partial x_i}$.

Let us denote by $A_N$ and $\overline{A}_N$ the shape operator of ${\Gamma}$ and $\overline{\Gamma}$, respectively. Let $e_1,\dots,e_{n-1}$ be the principal directions of $\Gamma$. Then, for each $1\leq i \leq n-1$, we have
$$ 
-\overline{A}_N (e_i)
=\left(\oconex_{{e}_i} {N}\right)^{\top\left({\overline{\Gamma}}\right)}
=\oconex_{{e}_i} {N}
=\conex_{{e}_i} {N}+\left(\oconex_{{e}_i} {N}\right)^{\bot(M)}
=\left(\conex_{{e}_i} {N}\right)^{\top\left({{\Gamma}}\right)}
=-A_N(e_i),$$
where the the penultimate step is true because $M$ is totally geodesic in $\oM$ and $N$ is a unit vector.
This means that $e_1,\dots,e_{n-1}$ are also principal directions of $\overline{\Gamma}$ at $x$. 
Since $Y$ is a tangent vector to $\overline{\Gamma}$ orthogonal to $e_i,\dots,e_{n-1}$, then $\rho^{-1}Y$ is the $n$-th principal direction of $\overline{\Gamma}$, and the corresponding principal curvature is given by
\begin{equation}
\kappa(x)=\rho^{-2}\escalar{\oconex_YY}{N}.
\label{eq:kappa}
\end{equation}
Indeed, 
\begin{align*}
\kappa(x)=\escalar{\overline{A}_N \left(\rho^{-1}Y\right)}{\rho^{-1}Y}_{\oM}
 = \rho^{-2}\escalar{-\oconex_{Y} \overline{N}}{Y}_{\oM}
 = \rho^{-2}\left(-Y\escalar{\overline{N}}{Y}_{\oM}  + \escalar{\overline{N}}{\oconex_{Y}{Y}}_{\oM}\right)
.
\end{align*}
Expression \eqref{eq:curv_oGamma} directly follows from the previous explanation. 
\end{proof}

%
%
%
%
%

\subsection{Transformation formulas}

In an open set in $M$ we define $w=\psi\circ \dd  + \varphi$, where $\psi\in\cl^{2}([a,b])$ and $\dd $ is a distance function (that is, $\nabla \dd =1$). 
We want to compute $\Q w$, where $\Q$ is the operator defined in \eqref{operador_Q}. 
We have
$$\nabla w= \psi'(\dd )\nabla \dd  +\nabla \varphi,$$
so 
$$\Hess w=\nabla \nabla w = \psi''\escalar{\nabla \dd }{\cdot}\escalar{\nabla \dd }{\cdot}+\psi'\Hess \dd  + \Hess \varphi$$
and
\begin{equation}
\Delta w=\tr \left(\Hess w\right) = \psi''+\psi'\Delta \dd  + \Delta \varphi.
\label{eq:Transf_Deltaw}
\end{equation}
We now compute
\begin{align*}
 &\Hess w (\nabla w,\nabla w)\\
=&\psi''\escalar{\nabla \dd }{\nabla w}^2+\psi'\Hess \dd  (\nabla w,\nabla w) + \Hess \varphi (\nabla w,\nabla w)\\
=&\psi''\escalar{\nabla \dd }{\nabla w}^2+\psi'\left( \Hess \dd  (\psi'\nabla \dd +\nabla \varphi,\psi'\nabla \dd +\nabla \varphi) \right) + \Hess \varphi (\nabla w,\nabla w)\\
=&\psi''\escalar{\nabla \dd }{\nabla w}^2+\psi'\left( \psi'^2\Hess \dd  (\nabla \dd ,\nabla \dd ) +2\psi'\Hess \dd  (\nabla \varphi,\nabla \dd )+\Hess \dd  (\nabla \varphi,\nabla \varphi) \right) + \Hess \varphi (\nabla w,\nabla w).
\end{align*}
Since
$$\Hess \dd  (X,\nabla \dd )=\escalar{\conex_X\nabla \dd }{\nabla \dd }=\frac{1}{2}X\escalar{\nabla \dd }{\nabla \dd }=0$$
for every smooth field $X$ over $M$, we conclude
\begin{equation}
\Hess w (\nabla w,\nabla w)=\psi''\escalar{\nabla \dd }{\nabla w}^2+\psi'\Hess \dd  (\nabla \varphi,\nabla \varphi) + \Hess \varphi (\nabla w,\nabla w).
\label{eq:Transf_Hessw}
\end{equation}

On the other hand, 
$$ \escalar{\cYY}{\nabla w}=\escalar{\cYY}{\psi'\nabla \dd  + \nabla \varphi}=\psi'\escalar{\cYY}{\nabla \dd } + \escalar{\cYY}{\nabla \varphi}.$$
Therefore,
\begin{equation}
\rho^{-2}\left(\rho^{-2}+W^2\right)\escalar{\cYY}{\nabla w}
=\rho^{-4}\escalar{\cYY}{\nabla w} + \psi' \rho^{-2} W^2 \escalar{\cYY}{\nabla \dd } + \rho^{-2} W^2\escalar{\cYY}{\nabla \varphi}.
\label{eq:Transf_TermoExtra}
\end{equation}

Using \eqref{eq:Transf_Deltaw}, \eqref{eq:Transf_Hessw} and \eqref{eq:Transf_TermoExtra} we finally get (see \eqref{operador_Q})
\begin{equation}\label{eq_transformacao_0}
\begin{split}
  \Q  w =\psi' W^2\Delta \dd  - \psi'\Hess \dd  (  \nabla \varphi, \nabla \varphi) + \psi'' W^2-\psi''\escalar{\nabla   \dd  }{\nabla  w }^2
 +    W^2\Delta \varphi  -    \Hess\varphi\left(\nabla  w ,\nabla  w \right)\\
- \rho^{-4}\escalar{\cYY}{\nabla  w } - \psi' \rho^{-2} W^2  \escalar{\cYY}{\nabla \dd }     - \rho^{-2}W^2\escalar{\cYY}{\nabla \varphi} -n H\left(\Phi(x,w)\right)W^3,
\end{split}
\end{equation}
where 
$$
W=W(x,\nabla w)=\sqrt{\rho^{-2}+\norm{\nabla w}^2}=\sqrt{\rho^{-2}+\norm{\psi'\nabla  \dd   +   \nabla \varphi}^2}$$

Suppose now that $\dd $ is the distance to a compact and embedded hypersurface $\Gamma$ in $M$, that is, $\dd (x)=\dist(x,\Gamma)$. 
{It is well known that $\dd $ is a $\cl^2$ function over the set
$$ \Sigma=\left\{\exp^{\bot}\left(y,tN(y)\right); \ y\in \Gamma, 0\leq {t}<\tau(y)\right\},$$
where $N$ is the unit normal field to $\Gamma$.  
Besides, every $x\in\Sigma$ is contained in an hypersurface $\Gamma_{\dd}$ which is parallel to a portion of $\Gamma$ (or even the whole $\Gamma$ in case $x$ is close enough to $\Gamma$). Furthermore, the laplacian of the function $\dd $ and the mean curvature of $\Gamma_{\dd}$ with respect to the unit normal field $N=\nabla \dd$ are related by the formula 
$$\Delta \dd (x) = -(n-1)\Hc_{\Gamma_{\dd }}(x).$$ 
}%

Let $\overline{\Gamma}$ the translation of $\Gamma$ along the flow lines of $Y$. Proposition \ref{Sup_transladadas} then yields
\begin{equation}
n\Hc_{\overline{\Gamma}_{\dd}}(x)= (n-1)\Hc_{\Gamma_{\dd }}(x) + \rho^{-2}\escalar{\cYY}{N} = - \Delta \dd (x) + \rho^{-2}\escalar{\cYY}{\nabla \dd }.
\label{eq:curv_cylinder}
\end{equation}

Substituting \eqref{eq:curv_cylinder} in \eqref{eq_transformacao_0} we finally get
\begin{equation}\label{eq_transformacao}
\begin{split}
  \Q  w =-n\psi' W^2\Hc_{\overline{\Gamma}_{\dd}} - \psi'\Hess \dd  (  \nabla \varphi, \nabla \varphi) + \psi'' W^2-\psi''\escalar{\nabla   \dd  }{\nabla  w }^2
 +    W^2\Delta \varphi  -    \Hess\varphi\left(\nabla  w ,\nabla  w \right)\\
- \rho^{-4}\escalar{\cYY}{\nabla  w }  - \rho^{-2}W^2\escalar{\cYY}{\nabla \varphi} -n H\left(\Phi(x,w)\right)W^3,
\end{split}
\end{equation}
Furthermore, if $\varphi$ is a constant, then
\begin{equation}\label{eq_transformacao_cte}
  \Q  w =-n\Hc_{\overline{\Gamma}_{\dd}}\psi' \left(\rho^{-2}+\psi'^2\right) + \rho^{-2}\psi'' - \rho^{-2}\psi'\kappa (x)  -n H\left(\Phi(x,w)\right)\left(\rho^{-2}+\psi'^2\right)^{3/2},
\end{equation}
where $\kappa(x)$ is the geodesic curvature of $\left\{\Phi_x(t),\ t\in\R\right\}$ at $x$ computed with respect to the normal $\nabla \dd (x)$.



\section{The a priori estimates}

{
Following a classical argument, the proof of Theorem \ref{T_exist_Ricci} depends on a $\cl^1$ a priori estimate for the solutions of related Dirichlet problems. In this section we establish the required estimate. 
}


\subsection{Global gradient estimate}

We adapt the proof of our Theorem 13 in \cite{alvarez2019existence}. 

\begin{teo}[Global gradient estimate]\label{teo_Est_global_gradiente}
Let $\W\subset M$ be a bounded domain. If a function $u\in\cl^3(\W)\cap\cl^1(\overline{\W})$ is a solution of \eqref{operador_minimo_1} for 
$H\in\cl^{1}\left(\ocyl\left({\sup\limits_{\overline{\W}}\modulo{u}}\right)\right)$ 
satisfying 
\begin{equation}
\escalar{\nH}{Y}_{\oM}\geq 0 \mbox{ in } \ocyl\left({\sup\limits_{\overline{\W}}\modulo{u}}\right),
\label{eq:HY_GlobalEstimate}
\end{equation}
then
$$ \sup_{\W}\norm{\nabla u}\leq C $$
where $C$ is a constant depending on {the data}
$n$, 
$\sup\limits_{\W}\modulo{\oRicc}$,  
$\sup\limits_{\W}\rho, \inf\limits_{\W}\rho$, 
$\sup\limits_{\W}\norm{\cYY}_{\oM}$, 
$\sup\limits_{\W\times\left[-\sup\limits_{\W}\modulo{u},\sup\limits_{\W}\modulo{u}\right]}\norm{D\Phi}$, 
$\sup\limits_{\ocyl\left({\sup\limits_{\overline{\W}}\modulo{u}}\right)}\norm{\overline{\nabla}H}_{\oM}$
and on $\sup\limits_{\W}\modulo{u}$ and 
$\sup\limits_{\partial\W}\norm{\nabla u}$.
\end{teo}

\begin{proof}

{As in the proof of Theorem 13 in \cite{alvarez2019existence}}, we set 
\begin{equation}
w(x)=\norm{\nabla u(x)}e^{Au(x)},
\label{function_w}
\end{equation}
where $A$ is a constant greater than 1 to be fixed latter on. We denote by $x_0$ a point where $w$ attains its maximum in $\overline{\W}$. Suppose that $\nabla u(x_0)\neq 0$. If $x_0\in\partial\W$, then
\begin{equation}\label{est_global_1}
\sup_{\W}\norm{\nabla u(x)}\leq\sup_{\partial\W}\norm{\nabla u}e^{2A\sup\limits_{\W}\modulo{u}}. 
\end{equation}

Let us assume, instead, that $x_0\in\W$. Set normal coordinates $(x_1,\dots,x_n)$ at $x_0$ in $M$ such that $\frac{\partial}{\partial x_1}\big|_{x_0}=\frac{\nabla u(x_0)}{\norm{\nabla u(x_0)}}$. Then, 
\begin{equation}
\nabla u(x_0)={\norm{\nabla u(x_0)}}\Eum\big|_{x_0}
\label{eq:grad}
\end{equation}
and
\begin{equation}\label{der_u_x0}
\Dk u(x_0)=\escalar{\Ek\big|_{x_0}}{\nabla u(x_0)}
=\norm{\nabla u(x_0)}\delta_{k1}.
\end{equation}


{We denote by $H_0=H(\Phi(x_0,u(x_0)))$, $\rho_0=\rho(x_0)$ e $W_0=\sqrt{\rho_0^{-2}+\norm{\nabla u(x_0)}^2}$. Differentiating \eqref{operador_minimo_1} with respect to $x_1$ we have}
\begin{equation}\label{der_MCE}
\begin{split}
n\left(\Dum H(\Phi(x,u(x)))\right)W^3 + nH(\Phi(x,u(x)))\left(\Dum \, W^3\right)=
\left(\Dum \, W^2 \right)\Delta u+W^2\left(\Dum\Delta u\right)\\
-\Dum \Hess u (\nabla u,\nabla u) 
-\Dum \left(\rho^{-2}\left(\rho^{-2}+W^2\right)\right)\escalar{\cYY}{\nabla u}
-\rho^{-2}\left(\rho^{-2}+W^2\right)\Dum\escalar{\cYY}{\nabla u}.
\end{split}
\end{equation}

Due to the fact that $u$ is a function taking values on the base $M$, the derivatives involving this function are obtained as in the proof of Theorem 13 in our paper \cite{alvarez2019existence}. 
So, for each $k\in\{1,2,\dots n\}$, one has
\begin{equation}\label{derk_2}
\Dumk u(x_0)=-A\norm{\nabla u (x_0)}^2\delta_{k1}.
\end{equation}
and 
\begin{equation}\label{Dknorm_2}
\Dk \left(\norm{\nabla u}^2\right)(x_0)=-2A\norm{\nabla u (x_0)}^3\delta_{k1}.
\end{equation}
On the other hand, since $\rho=\norm{Y}_{\oM}$ and $Y$ is a Killing field it follows
$$ \Di \rho^2 =2\escalar{\oconex_{\Ei}Y}{Y}_{\oM}=-2\escalar{\cYY}{\Ei}_{\oM}=-2\escalar{\cYY}{\Ei},$$
thus
\begin{equation}
\Di \rho^{-2} =2\rho^{-4}\escalar{\cYY}{\Ei}.
\label{eq:Di_rho}
\end{equation}
Expressions \eqref{Dknorm_2} and \eqref{eq:Di_rho} immediately give
\begin{equation}\label{dW2}
\Dum \, W^2 (x_0)
=2\rho_0^{-4}\escalar{\cYY}{\Eum}-2A\norm{\nabla u (x_0)}^3,
\end{equation}
and
\begin{equation}\label{dW3}
\Dum \, W^3 (x_0)=3\rho_0^{-4}W_0\escalar{\cYY}{\Eum}-3AW_0\norm{\nabla u (x_0)}^3.
\end{equation}

From \eqref{eq:Di_rho} and \eqref{dW2} it also follows
\begin{align*}
\Dum \left( \rho^{-2}\left(\rho^{-2}+W^2\right)\right)=&\left(\Dum \rho^{-2}\right)\left(\rho^{-2}+W^2\right) + \rho^{-2} \left(\Dum \rho^{-2}+ \Dum W^2\right)\\
=& 2\rho^{-4}\escalar{\cYY}{\Eum}\left(\rho^{-2}+W^2\right) \\
&+\rho^{-2} \left( 2\rho^{-4}\escalar{\cYY}{\Eum} + 2\rho^{-4}\escalar{\cYY}{\Eum}-2A\norm{\nabla u (x_0)}^3 \right).
\end{align*}
After some computations we get 
\begin{equation}
\Dum \left( \rho^{-2}\left(\rho^{-2}+W^2\right)\right)(x_0)= 2\rho_0^{-4}\left(3\rho_0^{-2}+W_0^2\right)\escalar{\cYY}{\Eum}
-2A\rho_0^{-2}\norm{\nabla u (x_0)}^3.
\label{eq:DumTextra_1}
\end{equation}

Furthermore,
\begin{align*} 
\Dum \escalar{\cYY}{\nabla u}=& \escalar{\conex_{\Eum}\cYY}{\nabla u} + \escalar{\cYY}{\conex_{\Eum}\nabla u}
\end{align*}
Since $\cYY$ is parallel the first term vanishes.
Recalling 
the properties of the normal coordinates at $x_0$ 
and using \eqref{derk_2} we obtain
\begin{equation}
\conex_{\Eum}\nabla u = \ds\sum_{i}  \Dumi u(x_0)\Ei = \ds\sum_i \left( -A\norm{\nabla u (x_0)}^2\delta_{i1} \right)\Ei = -A\norm{\nabla u (x_0)}^2\Eum.
\label{eq:conex_gradu}
\end{equation}
Hence,
\begin{equation}
\Dum \escalar{\cYY}{\nabla u}=-A\norm{\nabla u (x_0)}^2 \escalar{\cYY}{\Eum}.
\label{eq:DumTextra_2}
\end{equation}

From \eqref{eq:grad}, \eqref{eq:DumTextra_1} and \eqref{eq:DumTextra_2} we conclude
\begin{align*}
 &\Dum \left(\rho^{-2}\left(\rho^{-2}+W^2\right)\right)\escalar{\cYY}{\nabla u}
+\rho^{-2}\left(\rho^{-2}+W^2\right)\Dum\escalar{\cYY}{\nabla u}\\
 =&\left( 2\rho_0^{-4}\left(3\rho_0^{-2}+W_0^2\right)\escalar{\cYY}{\Eum}
-2A\rho_0^{-2}\norm{\nabla u (x_0)}^3   \right)\norm{\nabla u(x_0)}\escalar{\cYY}{\Eum}\\
&+\rho^{-2}\left(\rho^{-2}+W^2\right)\left(-A\norm{\nabla u (x_0)}^2 \escalar{\cYY}{\Eum}\right)\\
=& 2\rho_0^{-4}\left(3\rho_0^{-2}+W_0^2\right)\norm{\nabla u(x_0)}\escalar{\cYY}{\Eum}^2
-2A\rho_0^{-2}\norm{\nabla u (x_0)}^4   \escalar{\cYY}{\Eum}\\
&-A\rho^{-2}\left(\rho^{-2}+W^2\right)\norm{\nabla u (x_0)}^2 \escalar{\cYY}{\Eum},
\end{align*}
thus
{\small
\begin{equation}
\begin{split}
&\Dum \left(\rho^{-2}\left(\rho^{-2}+W^2\right)\right)\escalar{\cYY}{\nabla u}
+\rho^{-2}\left(\rho^{-2}+W^2\right)\Dum\escalar{\cYY}{\nabla u}\\
=&2\rho_0^{-4}\left(3\rho_0^{-2}+W_0^2\right)\norm{\nabla u(x_0)}\escalar{\cYY}{\Eum}^2
-A\rho^{-2}\norm{\nabla u (x_0)}^2\left(3\norm{\nabla u (x_0)}^2+2\rho^{-2}\right) \escalar{\cYY}{\Eum},
\end{split}
\label{eq:deriv_TermoExtra}
\end{equation}
}%

In these normal coordinates we also have
\begin{equation}\label{DumHessiano_antigo}
\Dum \Hess u \left( \Ei, \Ej \right) (x_0) =\Dumij u(x_0) - \escalar{\conex_{\Eum}\conex_{\Ei}\Ej}{\nabla u(x_0)}
\end{equation}
and 
\begin{equation}\label{DumLaplaciano}
\Dum\Delta u(x_0)=\ds\sum_{i=1}^n\left(\Dumii u(x_0) - \escalar{\conex_{\Eum}\conex_{\Ei}\Ei}{\nabla u(x_0)}\right).
\end{equation}
Besides, 
\begin{equation}\label{est_Dumumum}
\Dumumum u(x_0)\leq 2A^2\norm{\nabla u (x_0)}^3+\norm{\nabla u(x_0)}\escalar{\conex_{\Eum}\conex_{\Eum} \Eum}{\Eum}.
\end{equation}
\noindent and 
\begin{equation}\label{est_Dkkum}
\Dkkum u(x_0)\leq -A\norm{\nabla u (x_0)}\Dkk u(x_0)+\norm{\nabla u(x_0)}\escalar{\conex_{\Ek}\conex_{\Ek} \Eum}{\Eum} \ \mbox{ if }\  k>1.
\end{equation}
The explicit calculations to obtain \eqref{DumHessiano_antigo}, \eqref{DumLaplaciano}, \eqref{est_Dumumum} and \eqref{est_Dkkum} can be found in the previous work \cite{alvarez2019existence}.

From \eqref{DumHessiano_antigo} we derive
{\small
\begin{equation}
\Dum \Hess u(\nabla u, \nabla u)({x_0})=-2A \norm{\nabla u(x_0)}^3\Dumum u
+ \norm{\nabla u(x_0)}^2\left(\Dumumum u (x_0) -\escalar{\conex_{\Eum}\conex_{\Eum}\Eum}{\nabla u(x_0)}\right). 
\label{DumHessiano}
\end{equation}
}
Hence, recalling again \eqref{eq:grad}, we obtain
\begin{align*}
 &W_0^2\Dum\Delta u(x_0)-\Dum \Hess u(\nabla u, \nabla u)({x_0})\\
=&W_0^2 \ds\sum_{i=1}^n\left(\Dumii u(x_0) - \escalar{\conex_{\Eum}\conex_{\Ei}\Ei}{\nabla u(x_0)}\right)\\
&+2A \norm{\nabla u(x_0)}^3\Dumum u
-\norm{\nabla u(x_0)}^2\left(\Dumumum u (x_0) -\escalar{\conex_{\Eum}\conex_{\Eum}\Eum}{\nabla u(x_0)}\right)\\
=&W_0^2 \ds\sum_{i>1}^n\left(\Dumii u(x_0) - \escalar{\conex_{\Eum}\conex_{\Ei}\Ei}{\nabla u(x_0)}\right)\\
&+2A \norm{\nabla u(x_0)}^3\Dumum u
+\rho^{-2}\left(\Dumumum u (x_0) -\escalar{\conex_{\Eum}\conex_{\Eum}\Eum}{\nabla u(x_0)}\right)\\
\leq&W_0^2 \ds\sum_{i>1}^n\left( -A\norm{\nabla u (x_0)}\Dii u(x_0)+\norm{\nabla u(x_0)}\left(\escalar{\conex_{\Ei}\conex_{\Ei} \Eum}{\Eum} - \escalar{\conex_{\Eum}\conex_{\Ei}\Ei}{\Eum}\right)\right)\\
&+2A \norm{\nabla u(x_0)}^3\Dumum u
+\rho^{-2}\left(2A^2\norm{\nabla u (x_0)}^3+\escalar{\conex_{\Eum}\conex_{\Eum} \Eum}{\nabla u(x_0)}-\escalar{\conex_{\Eum}\conex_{\Eum}\Eum}{\nabla u(x_0)}\right)\\
%
%
=&-A\norm{\nabla u (x_0)}W_0^2 \ds\sum_{i>1}^n\Dii u(x_0)+\norm{\nabla u(x_0)}W_0^2 \sum_{i>1}^nR\left(\Ei,\Eum,\Ei,\Eum\right)\\
&+2A \norm{\nabla u(x_0)}^3\Dumum u +2A^2\rho^{-2}\norm{\nabla u (x_0)}^3,
\end{align*}
where $R$ is the curvature of $M$. 

{We notice now that $(x_1,\dots,x_n,z)$ are coordinates in a neighbourhood of $x_0$ in $\oM$}. Denoting by $\overline{R}$ the curvature of $\overline{M}$ we have
$$ R\left(\Ei,\Eum,\Ei,\Eum\right)= \overline{R}\left(\Ei,\Eum,\Ei,\Eum\right)$$
because the second fundamental form of $M$ vanishes. In addition, since $Y$ is normal to $M$ and $\oconex_{Y}Y$ is parallel, we infer 
\begin{align*}
\overline{R}\left(Y,\Eum,Y,\Eum\right)=&\escalar{\oconex_{Y}\oconex_{\Eum} Y}{\Eum} - \escalar{\oconex_{\Eum}\oconex_{Y}Y}{\Eum}=0. 
\end{align*}
Then, 
\begin{equation}
\sum_{i>1}^nR\left(\Ei,\Eum,\Ei,\Eum\right)=-\oRicc(\Eum).
\label{eq:Ricci}
\end{equation}
Consequently,
\begin{equation}\label{SegundoTermo}
\begin{split}
W_0^2\Dum\Delta u(x_0)-\Dum \Hess u(\nabla u, \nabla u)({x_0})
=-A\norm{\nabla u (x_0)}W_0^2 \ds\sum_{i>1}^n\Dii u(x_0)-\norm{\nabla u(x_0)}W_0^2 \oRicc(\Eum)\\
+2A \norm{\nabla u(x_0)}^3\Dumum u +2A^2\rho^{-2}\norm{\nabla u (x_0)}^3.
\end{split}
\end{equation}

\bigskip

Finally,
{\small
\begin{equation}
\Dum H(\Phi(x,u(x)))=\escalar{\nH(\Phi(x_0,u(x_0)))}{\Dum \Phi(x_0,u(x_0))}_{\oM}
+\escalar{\nH(\Phi(x_0,u(x_0)))}{Y_{\Phi(x_0,u_0)}}_{\oM}\norm{\nabla u(x_0)},
\label{eq:Deriv_HoPHI}
\end{equation}
}%

Noticing that in normal coordinates 
\begin{equation}\label{laplac_x0}
\Delta u(x_0)=\ds\sum_{i=1}^n \Dii u(x_0)
\end{equation}
and substituting \eqref{dW2}, \eqref{dW3}, \eqref{eq:deriv_TermoExtra}, \eqref{SegundoTermo} and \eqref{eq:Deriv_HoPHI}, in \eqref{der_MCE} it follows
\begin{align*}
&nW_0^3\escalar{\nH(\Phi(x_0,u(x_0)))}{\Dum \Phi(x_0,u(x_0))}_{\oM}
+nW_0^3\norm{\nabla u(x_0)}\escalar{\nH(\Phi(x_0,u(x_0)))}{Y_{\Phi(x_0,u_0)}}_{\oM}\\
&+3n\rho_0^{-4}H_0W_0\escalar{\cYY}{\Eum}-3nA H_0 W_0\norm{\nabla u(x_0)}^3\\[1em]
\leq&\left(2\rho_0^{-4}\escalar{\cYY}{\Eum}-2A\norm{\nabla u(x_0)}^3\right)\Delta u(x_0)
-A\norm{\nabla u (x_0)}W_0^2 \ds\sum_{i>1}^n\Dii u(x_0)\\&-\norm{\nabla u(x_0)}W_0^2 \oRicc(\Eum)+2A \norm{\nabla u(x_0)}^3\Dumum u +2A^2\rho_0^{-2}\norm{\nabla u (x_0)}^3\\
&-2\rho_0^{-4}\left(3\rho_0^{-2}+W_0^2\right)\norm{\nabla u(x_0)}\escalar{\cYY}{\Eum}^2
+A\rho_0^{-2}\norm{\nabla u (x_0)}^2\left(3\norm{\nabla u (x_0)}^2+2\rho_0^{-2}\right) \escalar{\cYY}{\Eum}\\
=&2\rho_0^{-4}\escalar{\cYY}{\Eum}\Dumum u(x_0)
+\left(2\rho_0^{-4}\escalar{\cYY}{\Eum}-2A\norm{\nabla u(x_0)}^3-A\norm{\nabla u (x_0)}W_0^2 \right)\ds\sum_{i>1}^n\Dii u(x_0)
\\&-\norm{\nabla u(x_0)}W_0^2 \oRicc(\Eum)+2A^2\rho_0^{-2}\norm{\nabla u (x_0)}^3\\
&-2\rho_0^{-4}\left(3\rho_0^{-2}+W_0^2\right)\norm{\nabla u(x_0)}\escalar{\cYY}{\Eum}^2
+A\rho_0^{-2}\norm{\nabla u (x_0)}^2\left(3\norm{\nabla u (x_0)}^2+2\rho_0^{-2}\right) \escalar{\cYY}{\Eum}.
\end{align*}
Since \eqref{derk_2} holds we have
\begin{equation}
\begin{split}
&nW_0^3\escalar{\nH(\Phi(x_0,u(x_0)))}{\Dum \Phi(x_0,u(x_0))}_{\oM}
+nW_0^3\norm{\nabla u(x_0)}\escalar{\nH(\Phi(x_0,u(x_0)))}{Y_{\Phi(x_0,u_0)}}_{\oM}\\
\leq&3nA H_0 W_0\norm{\nabla u(x_0)}^3+\left(-3n\rho_0^{-4}H_0W_0+3A\rho_0^{-2}\norm{\nabla u(x_0)}^4\right)\escalar{\cYY}{\Eum}\\
&+\left(2\rho_0^{-4}\escalar{\cYY}{\Eum}-A\norm{\nabla u(x_0)}\left(\rho^{-2}+3\norm{\nabla u (x_0)}^2 \right)\right)\ds\sum_{i>1}^n\Dii u(x_0)\\
&-\norm{\nabla u(x_0)}W_0^2 \oRicc(\Eum)+2A^2\rho_0^{-2}\norm{\nabla u (x_0)}^3
-2\rho_0^{-4}\left(3\rho_0^{-2}+W_0^2\right)\norm{\nabla u(x_0)}\escalar{\cYY}{\Eum}^2.
\end{split}
\label{eq:est_part1}
\end{equation}


{In order to find the value of $\sum\limits_{i>1}\Dii u$ we evaluate the mean curvature equation \eqref{operador_minimo_1} at $x_0$. Recall also that, in normal coordinates, 
\begin{equation}\label{hess_x0}
\Hess u\left(\Ei,\Ej\right) (x_0)=\Dij u(x_0). 
\end{equation}
From \eqref{eq:grad} and \eqref{hess_x0} one has
\begin{equation}
\Hess u (\nabla u, \nabla u)(x_0)=\norm{\nabla u(x_0)}^2 \Hess u(x_0)\left(\Eum,\Eum\right)=\norm{\nabla u(x_0)}^2 \Dumum u(x_0)
\label{eq:Hess_0}
\end{equation}
\noindent Substituting \eqref{laplac_x0} and \eqref{eq:Hess_0} in the mean curvature equation \eqref{operador_minimo_1} it follows
\begin{align*}
W_0^2\ds\sum_{i=1}^n \Dii u(x_0) -\norm{\nabla u(x_0)}^2\Dumum u(x_0)-\rho_0^{-2}\left(\rho_0^{-2}+W_0^2\right)\escalar{\cYY}{\nabla u (x_0)} = nH_0W_0^3.
\end{align*}
Using \eqref{eq:grad} and \eqref{derk_2}, and making some algebraic computations we conclude
\begin{equation}\label{equ_curv_media_MxR_Delta_x0}
\ds\sum_{i>1}\Dii u(x_0) = nH_0W_0+\dfrac{A\rho_0^{-2}\norm{\nabla u(x_0)}^2}{W_0^2} + \dfrac{\rho_0^{-2}\left(\rho_0^{-2}+W_0^2\right)\norm{\nabla u (x_0)}}{W_0^2}\escalar{\cYY}{\Eum}.
\end{equation}
}

We can compute now the term in \eqref{eq:est_part1} involving this quantity:
\begin{align*}
&\left(2\rho_0^{-4}\escalar{\cYY}{\Eum}-A\norm{\nabla u(x_0)}\left(\rho^{-2}+3\norm{\nabla u (x_0)}^2 \right)\right)\ds\sum_{i>1}^n\Dii u(x_0)\\
=
&2\rho_0^{-4}\left(nH_0W_0+\dfrac{A\rho_0^{-2}\norm{\nabla u(x_0)}^2}{W_0^2}\right)\escalar{\cYY}{\Eum} 
+\dfrac{2\rho_0^{-6}(\rho_0^{-2}+W_0^2)\norm{\nabla u(x_0)}}{W_0^2}\escalar{\cYY}{\Eum}^2\\
&-AnH_0\norm{\nabla u(x_0)}W_0\left(\rho^{-2}+3\nabla u(x_0)^2\right)-\dfrac{A^2\rho^{-2}\norm{\nabla u(x_0)}^3\left(\rho^{-2}+3\nabla u(x_0)^2\right)}{W_0^2}\\
&-\dfrac{A\rho^{-2}\left(\rho^{-2}+3\nabla u(x_0)^2\right)\left(\rho_0^{-2}+W_0^2\right)\norm{\nabla u (x_0)}^2}{W_0^2}
\escalar{\cYY}{\Eum}\\
=
&\dfrac{\rho_0^{-2}}{W_0^2}\left(2\rho_0^{-2}nH_0W_0^3+2A\rho_0^{-4}\norm{\nabla u(x_0)}^2 
-A\norm{\nabla u(x_0)}^2\left(2\nabla u(x_0)^2+W_0^2\right)\left(\rho_0^{-2}+W_0^2\right)\right)\escalar{\cYY}{\Eum}\\
&+\dfrac{2\rho_0^{-6}(\rho_0^{-2}+W_0^2)\norm{\nabla u(x_0)}}{W_0^2}\escalar{\cYY}{\Eum}^2\\
&-AnH_0\norm{\nabla u(x_0)}W_0\left(\rho^{-2}+3\nabla u(x_0)^2\right)-\dfrac{A^2\rho^{-2}\norm{\nabla u(x_0)}^3\left(\rho^{-2}+3\nabla u(x_0)^2\right)}{W_0^2}\\
\end{align*}


We substitute this last expression in \eqref{eq:est_part1} to infer
\begin{equation}\label{Part_2}
\begin{split}
&nW_0^3\escalar{\nH(\Phi(x_0,u(x_0)))}{\Dum \Phi(x_0,u(x_0))}_{\oM}
+nW_0^3\norm{\nabla u(x_0)}\escalar{\nH(\Phi(x_0,u(x_0)))}{Y_{\Phi(x_0,u_0)}}_{\oM}\\
\leq
&f\dfrac{\rho_0^{-2}}{W_0^2}\escalar{\cYY}{\Eum}+g\dfrac{2\rho_0^{-4}}{W_0^2}\escalar{\cYY}{\Eum}^2 
-An\rho^{-2}H_0\norm{\nabla u(x_0)}W_0\\&
-\dfrac{A^2\rho^{-2}\norm{\nabla u(x_0)}^3\left(\norm{\nabla u(x_0)}^2-\rho^{-2}\right)}{W_0^2}
-\norm{\nabla u(x_0)}W_0^2 \oRicc(\Eum)
\end{split}
\end{equation}
where
\begin{equation*}
\begin{split}
f=&-3n\rho_0^{-2}H_0W_0^3+3A\norm{\nabla u(x_0)}^4W_0^2+2\rho_0^{-2}nH_0W_0^3+2A\rho_0^{-4}\norm{\nabla u(x_0)}^2 \\
&-A\norm{\nabla u(x_0)}^2\left(2\nabla u(x_0)^2+W_0^2\right)\left(\rho_0^{-2}+W_0^2\right)\\
=&-\rho_0^{-2}\left(4A\norm{\nabla u(x_0)}^4+nH_0W_0^3\right)
\end{split}
\end{equation*}
and
\begin{equation*}
g=\rho_0^{-2}(\rho_0^{-2}+W_0^2)-(3\rho_0^{-2}+W_0^2)W_0^2\leq 0. 
\end{equation*}
Substituting this facts in equation \eqref{Part_2}, dividing by $\rho^{-2}A^2W_0^3$ and rearranging terms we get
\begin{align*}
&\dfrac{\norm{\nabla u(x_0)}^3\left(\norm{\nabla u(x_0)}^2-\rho^{-2}\right)}{W_0^5}\\
\leq
&-\dfrac{n}{A^2\rho^{-2}}\norm{\nabla u(x_0)}\escalar{\nH(\Phi(x_0,u(x_0)))}{Y_{\Phi(x_0,u_0)}}_{\oM}
-\dfrac{n}{A^2\rho^{-2}}\escalar{\nH(\Phi(x_0,u(x_0)))}{\Dum \Phi(x_0,u(x_0))}_{\oM}
\\
&-\rho_0^{-2}\left(4\dfrac{\norm{\nabla u(x_0)}^4}{AW_0^5}+\dfrac{nH_0}{A^2W_0^5}\right)\escalar{\cYY}{\Eum}
-\dfrac{nH_0\norm{\nabla u(x_0)}}{A^2W_0^2}-\dfrac{\norm{\nabla u(x_0)}}{A^2\rho_0^{-2}W_0} \oRicc(\Eum)
\end{align*}

Due to the definition of $W$, we can choose $A$ large enough to guarantee that the sum of the second, third, fourth and fifth terms in the right hand of the inequality above is strictly less than $\frac{1}{8}$. 
Besides, hypothesis \eqref{eq:HY_GlobalEstimate} implies that the first term is non positive. 
Since $W^{3}<4\norm{\nabla u(x_0)}^3$ and we can assume that $\norm{\nabla u(x_0)}>\rho^{-1}$ we have
$$ \dfrac{\norm{\nabla u(x_0)}^2-\rho^{-2}}{\norm{\nabla u(x_0)}^2+\rho^{-2}}<\dfrac{1}{2}$$
which implies
$$ \rho\norm{\nabla u(x_0)}<\sqrt{3}.$$
Recalling the definition of $w$ in \eqref{function_w} we conclude
\begin{equation}\label{est_global_grad_interior}
\sup_{\W}\norm{\nabla u}\leq\sqrt{3}\ds\sup_{\W}\rho^{-1}e^{2A\sup\limits_{\W}\modulo{u}}, 
\end{equation}
where $A$ depends on 
$n$, 
$\sup\limits_{\W}\modulo{\oRicc}$,  
$\sup\limits_{\W}\rho, \inf\limits_{\W}\rho$, 
$\sup\limits_{\W}\norm{\cYY}_{\oM}$, 
$\sup\limits_{\W\times\left[-\sup\limits_{\W}\modulo{u},\sup\limits_{\W}\modulo{u}\right]}\norm{D\Phi}$, 
$\sup\limits_{\cyl_{\sup\limits_{\W}\modulo{u}}}\norm{\overline{\nabla}H}_{\oM}$. 

\bigskip
From \eqref{est_global_1} and \eqref{est_global_grad_interior} we obtain the desired estimate.
\end{proof}


\begin{obs}\label{GlobalEstimate_family}
This estimate also holds for every solution of the family of equations $\Q_{\sigma} u =0$ in $\W$. 
\end{obs}

\subsection{Boundary gradient estimates}\label{sec_est_grad_bordo}
As we observe in Theorem \ref{teo_Est_global_gradiente}, in order to have a global gradient estimate, a priori estimates for the gradient at the boundary is required. 

\begin{teo}\label{teo_Est_gradiente_fronteira}
Let $\W\in M$ be a $\cl^2$ bounded domain and $\varphi\in\cl^2(\overline{\W})$. 
Let 
$H\in\cl^{1}\left({\ocyl\left({\sup\limits_{\overline{\W}}\modulo{u}}\right)}\right)$ 
satisfying 
\begin{equation}
\escalar{\nH}{Y}_{\oM}\geq 0 \mbox{ in } \ocyl,      
\label{eq:HY}
\end{equation}
and
\begin{equation}\label{cond_H_Ricci_sup_est_grad_fronteira}
\norm{\nH\left(\Phi(x,\varphi(y(x)))\right)}_{\oM}\leq \left(H\left(\Phi(x,\varphi(y(x)))\right)\right)^2 + \frac{\oRicc_x}{n} \ \ \forall \ x\in\W_0.
\end{equation} 
Suppose also that 
\begin{equation}\label{cond_Serrin}
\Hc_{\cyl}(y)\geq \modulo{H(\Phi(y,\varphi(y)))} \ \forall \ y\in\partial\W.
\end{equation}
Then, if $u\in\cl^2(\W)\cap\cl^1(\overline{\W})$ is a solution of \eqref{operador_minimo_1} and $u=\varphi$ in $\partial\W$, we have
$$
\sup\limits_{\partial\W}\norm{\nabla u}\leq C,
$$
where $C$ depends on $n$, $\W$, $\min\limits_{\W}\rho$, $\max\limits_{\W}\rho$, $\sup\limits_{\W}\norm{\cYY}$, $\norm{\varphi}_1$, $\norm{\varphi}_2$, 
$$\sup\limits_{\ocyl\left({\sup\limits_{\overline{\W}}\modulo{\varphi}}\right)}\modulo{H},\ 
\sup\limits_{\ocyl\left({\sup\limits_{\overline{\W}}\modulo{\varphi}}\right)}\norm{\nH},\  
\sup\limits_{\W\times\left[-\sup\limits_{\W}\modulo{\varphi},\sup\limits_{\W}\modulo{\varphi}\right]}\norm{D\Phi}.$$
 
\end{teo}
\begin{proof}
It is well known that there exists $\tau>0$ such that 
$$\dd \in\cl^2\left(\left\{x\in\W; \dd (x)<\tau \right\}\right).$$
Let us fix $a<\tau$ to be precised later, and consider 
$$\W_{a}=\left\{x\in M; \dd (x)<a \right\} .$$

Let $w = \psi\circ \dd  +  \varphi$, where $\psi\in\cl^2([0,\tau])$ and satisfies $\psi\geq 0$, $\psi'(t)\geq 1$, $\psi''(t) \leq 0$ and $t\psi'(t)\leq 1$.
%
%
Using equation \eqref{eq_transformacao} we have 
\begin{equation*}
\begin{split}
\Q  w =-n\psi'W ^2\Hc_{\cyl_{\dd }} -\psi'\Hess \dd  (  \nabla \varphi, \nabla \varphi) + \psi''W ^2-\psi''\escalar{\nabla   \dd  }{\nabla w }^2
+  W ^2\Delta \varphi -    \Hess\varphi\left(\nabla w ,\nabla w \right)\\
- \rho^{-4}\escalar{\cYY}{\nabla w }  - \rho^{-2}W ^2\escalar{\cYY}{\nabla \varphi}-n{H\left(\Phi(x,w )\right)}W ^3,
\end{split}
\end{equation*}
where 
$$
W =\sqrt{\rho^{-2}+\norm{\nabla w }^2}=\sqrt{\rho^{-2}+\norm{ \psi'\nabla  \dd   +   \nabla \varphi}^2}
$$
and $\cyl_{\dd}$ is the Killing cylinder parallel to $\cyl$ containing $x$. 

Note first that hypothesis \eqref{eq:HY} guaranties that $H\circ\Phi$ is increasing in the $z$ variable. So,
$$  H\left(\Phi\left(x,w (x)\right)\right) 
\geq  H\left(\Phi(x,\varphi(x))\right)\geq - \modulo{H\left(\Phi(x, \varphi(x))\right)}$$
because the sign of $\psi$ also implies $w\geq \varphi$. 
Thus,
\begin{equation*}
\begin{split}
\Q  w \leq-n\psi'W ^2\Hc_{\cyl_{\dd}} -\psi'\Hess \dd  (  \nabla \varphi, \nabla \varphi) + \psi''W ^2-\psi''\escalar{\nabla   \dd  }{\nabla w }^2
+  W ^2\Delta \varphi -    \Hess\varphi\left(\nabla w ,\nabla w \right)\\
- \rho^{-4}\escalar{\cYY}{\nabla w }  - \rho^{-2}W ^2\escalar{\cYY}{\nabla \varphi}+n\modulo{H\left(\Phi(x,\varphi)\right)}W ^3,
\end{split}
\end{equation*}
which is equivalent to 
\begin{equation}\label{eq_barreira_superior}
\begin{split}
\Q  w \leq-n\psi'W ^2\left( \Hc_{\cyl_{\dd}}-\modulo{H\left(\Phi(x,\varphi)\right)}\right) +n \modulo{H\left(\Phi(x, \varphi)\right)}W ^2 \left(W -\psi'\right)
-\psi'\Hess \dd  (  \nabla \varphi, \nabla \varphi) \\
+ \psi''W ^2-\psi''\escalar{\nabla   \dd  }{\nabla w }^2
+  W ^2\Delta \varphi -    \Hess\varphi\left(\nabla w ,\nabla w \right)
- \rho^{-4}\escalar{\cYY}{\nabla w }  - \rho^{-2}W ^2\escalar{\cYY}{\nabla \varphi},
\end{split}
\end{equation}

On the other hand, 
\begin{equation*}
\norm{\nabla w }^2
=\norm{ \psi' \nabla \dd   +  \nabla \varphi}^2
=\left(\psi'^2+2 \psi'\escalar{+ \nabla \dd }{\nabla \varphi}+\norm{\nabla \varphi}^2\right)
\leq \left(1+\norm{\varphi}_1\right)^2\psi'^2,
\end{equation*}
hence
\begin{equation}\label{est_W2}
W ^2 =\rho^{-2}+\norm{\nabla w }^2
\leq\rho_i^{-2}+ \left(1+\norm{\varphi}_1\right)^2\psi'^2
\leq\left(\rho_i^{-2}+ \left(1+\norm{\varphi}_1\right)^2\right)\psi'^2
\leq  \left(1+\rho_i^{-1} +\norm{\varphi}_1\right)^2\psi'^2.
\end{equation}
Since $\dd $ is a distance function it follows
\begin{equation}\label{Wmenospsi_0}
W -\psi'= \sqrt{\rho^{-2}+\norm{\nabla w }^2}-\psi'\leq \rho^{-1}+\norm{ \psi'\nabla \dd  +\nabla \varphi} -\psi' 
\leq \rho_i^{-1} +\norm{\varphi}_1< 1+ \rho_i^{-1} +\norm{\varphi}_1.
\end{equation}
From \eqref{est_W2} and \eqref{Wmenospsi_0} we obtain 
\begin{equation}\label{Wmenospsi}
n \modulo{H\left(\Phi(x, \varphi(x))\right)} \left(W -\psi'\right)W ^2\leq n h_0\left(1+\rho_i^{-1}+\norm{\varphi}_1\right)^3\psi'^2,
\end{equation}
where
$$h_0=\sup\limits_{\ocyl\left({\sup\limits_{\overline{\W}}\modulo{\varphi}}\right)}\modulo{H}.$$

Also $\psi'\geq 1$ and $\dd $ is of class $\cl^2$, then
\begin{equation}\label{BBB}
\psi'\modulo{\Hess \dd  ( \nabla \varphi, \nabla \varphi)} \leq \norm{\dd }_2\norm{\varphi}_1^2\psi'^2 .
\end{equation}

Recalling that $\psi''<0$ we have
\begin{equation}\label{AAA}
\psi''W ^2-\psi''\escalar{\nabla   \dd  }{\nabla w }^2
\leq\psi''W ^2-\psi''\norm{\nabla w }^2
=\rho^{-2}\psi''\leq\rho_i^{-2}\psi''.
\end{equation}

From \eqref{est_W2} and the regularity of $\varphi$ we also have
%
\begin{equation}\label{CCC}
\modulo{  W ^2\Delta \varphi -    \Hess \varphi\left(\nabla w ,\nabla w \right)} 
 \leq  n \norm{\varphi}_2 W ^2  + \norm{\varphi}_2 \norm{\nabla w }^2
 < 2 n \norm{\varphi}_2 \left(1+\rho_i^{-1}+\norm{\varphi}_1\right)^2\psi'^2 .
\end{equation}

Furthermore, 
\begin{align*}
\modulo{- \rho^{-4}\escalar{\cYY}{\nabla w }  - \rho^{-2}W ^2\escalar{\cYY}{\nabla \varphi}}
\leq&\rho^{-4}\norm{\cYY}\norm{\nabla w } +\rho^{-2}\norm{\cYY}\norm{\nabla \varphi}W ^2\\
=&\rho^{-2}\norm{\cYY}\left(\rho^{-2}\norm{\nabla w } +\norm{\nabla \varphi}W ^2\right)\\
\leq &\rho^{-2}\norm{\cYY}\left(\rho^{-1}\frac{\rho^{-2}+\norm{\nabla w }^2}{2} +\norm{\nabla \varphi}W ^2\right)\\
= &\rho^{-2}\norm{\cYY}\left(\frac{\rho^{-1}}{2} +\norm{\nabla \varphi}\right)W ^2\\
< &\rho^{-2}\norm{\cYY}\left(1+\rho^{-1} +\norm{\nabla \varphi}\right)W ^2.
\end{align*}
From \eqref{est_W2} we conclude
\begin{equation}
\modulo{- \rho^{-4}\escalar{\cYY}{\nabla w }  - \rho^{-2}W ^2\escalar{\cYY}{\nabla \varphi}}
\leq\rho_i^{-2}\sup\limits_{\W}\norm{\cYY}\left(1+\rho_i^{-1}+\norm{\varphi}_1\right)^3\psi'^2.
\label{eq:termo_extra_bordo}
\end{equation}

Finally we estimate the first term in \eqref{eq_barreira_superior}. 
Since the hypothesis \eqref{cond_H_Ricci_sup_est_grad_fronteira} and \eqref{cond_Serrin} hold, necessarily the mean curvature of the Killing cylinders $\cyl_{d}$ parallel to $\cyl$ is increasing as a function of the distance along the inner normal geodesics to the boundary, $\gamma_y(t)=\exp_{y}(tN_y)$ with $y\in\partial\W$ and $0\leq t \leq \tau$. This crucial geometric property directly follows from Lemma 6 in \cite{alvarez2019existence}. 
Indeed, let $h(t)=H\left(\Phi(\gamma_y(t),\varphi(y))\right)$. Then 
$$\Hc_{\cyl}(y) \geq 
\modulo{h(0)}$$
because \eqref{cond_Serrin} holds. Besides, from hypothesis \eqref{cond_H_Ricci_sup_est_grad_fronteira} we conclude that
\begin{align*}
\modulo{h'(t)}
		 =&\modulo{\escalar{\nH\left(\Phi\left(\gamma_y(t),\varphi(y)\right)\right)}{\left(\dd \Phi_{\varphi(y)}\right)_{\gamma_y(t)}\left(\gamma_y'(t)\right)}_{\oM}}\\
		 \leq&\norm{\nH\left(\Phi\left(\gamma_y(t),\varphi(y)\right)\right)}_{\oM}\norm{\left(\dd \Phi_{\varphi(y)}\right)_{\gamma_y(t)}\left(\gamma_y'(t)\right)}_{\oM}\\
		 =&\norm{\nH\left(\Phi\left(\gamma_y(t),\varphi(y)\right)\right)}_{\oM}\norm{\gamma_y'(t)}_{\oM}\\
		 =&\norm{\nH\left(\Phi\left(\gamma_y(t),\varphi(y)\right)\right)}_{\oM}\\
		\leq& \left(H\left(\Phi\left(\gamma_y(t),\varphi(y)\right)\right)\right)^2+\dfrac{\oRicc_{\gamma_{y}(t)}\left(\gamma_y'(t)\right)}{n}\\
		= & \left(h(t)\right)^2 + \dfrac{\oRicc_{\gamma_{y}(t)}\left(\gamma_y'(t)\right)}{n}
\end{align*}
where we have used that $\Phi_{\varphi(y)}$ is an isometry and that $\gamma_y$ is also a normalized geodesic in $\overline{M}$. 
Therefore,
\begin{equation}\label{paraObs}
\Hc_{\cyl_{\dd (x)}}(x)\geq \Hc_{\cyl}(y(x)).
\end{equation}

Using \eqref{cond_Serrin} again and \eqref{paraObs} we conclude
$$\Hc_{\cyl_{\dd}} - \modulo{H\left(\Phi(x,\varphi(x))\right)}
\geq \Hc_{\cyl}(y)- \modulo{H\left(\Phi(x,\varphi(x))\right)}
\geq \modulo{H\left(\Phi(y,\varphi(y))\right)} - \modulo{H\left(\Phi(x,\varphi(x))\right)}
$$
{
In addition, if we apply the mean value theorem to the function 
$$ f(t)= H(\Phi(\gamma_y(t),\varphi(\gamma_y(t)))) , \ t\in[0,d(x)]$$
we derive
$$\Big|{\modulo{H\left(\Phi(x,\varphi(x))\right)}-\modulo{H\left(\Phi(y,\varphi(y))\right)}}\Big|\leq h_1\Phi_1\rho_s^2(1+\rho_i^{-1}+\norm\varphi_1)^2\dd (x),$$
where
$$\rho_{M}=\max\limits_{\W}\rho,\ h_1=\sup\limits_{\ocyl\left(\sup\limits_{\W}\modulo{\varphi}\right)}\norm{\nH}_{\oM}
\ \mbox{ and }\ 
\Phi_1=\sup\limits_{\W\times\left[-\sup\limits_{\W}\modulo{\varphi},\sup\limits_{\W}\modulo{\varphi}\right]}\norm{D\Phi}.$$
%
Then,
\begin{equation}\label{FFF}
\Hc_{\cyl_{\dd}} - \modulo{H\left(\Phi(x,\varphi(x))\right)}
\geq - h_1\Phi_1\rho_s^2(1+\rho_i^{-1}+\norm\varphi_1)^2\dd (x).
\end{equation}
Expressions \eqref{est_W2} and \eqref{FFF} combined with assumption $t\psi'(t)\leq 1$ on $\psi$ yield
{\small%
\begin{equation}\label{lastTerm}
-n  \psi'W ^2\left(\Hc_{\cyl_{\dd}} - \modulo{H\left(\Phi(x,\varphi(x))\right)} \right)
\leq  n  h_1\Phi_1\rho_s^2(1+\rho_i^{-1}+\norm\varphi_1)^2\dd  \psi'(\dd ) W ^2 
\leq  n  h_1\Phi_1\rho_s^2\left(1+\rho_i^{-1}+\norm\varphi_1\right)^4\psi'^2.
\end{equation}}%
}


Using the estimates \eqref{Wmenospsi}, \eqref{BBB}, \eqref{AAA}, \eqref{CCC}, \eqref{eq:termo_extra_bordo} and \eqref{lastTerm} in \eqref{eq_barreira_superior} it follows
%
$$ \Q _{ } w  < \nu\psi'^2+\psi'',$$
where we have chosen $\nu$ as 
\begin{equation}\label{nu}
\nu= \left(1+1/\tau\right) \left(1+nh_0+\norm{\dd }_2+2n\norm{\varphi}_2 + \rho_i^{-2}\sup\limits_{\W}\norm{\cYY} + nh_1\Phi_1\rho_s^2 \right)\left(1+\rho_i^{-1}+\norm{\varphi}_1\right)^4.
\end{equation}

We now define
\begin{equation}\label{psi_est_grad_bordo}
\psi(t)=\dfrac{1}{\nu}\log(1+kt),
\end{equation}
where
\begin{equation}\label{cte_k}
k=\nu\ds e^{\nu(\norm{u}_0+\norm{\varphi}_0)}.
\end{equation}
Is not difficult to see that the function in \eqref{psi_est_grad_bordo} satisfies all the initial requirements if we choose
\begin{equation}\label{paraP1}
\psi'(a)=\dfrac{k}{\nu(1+ka)}=1. 
\end{equation}
Condition \eqref{paraP1} also implies 
$$
a
=\dfrac{e^{\nu(\norm{u}_0+\norm{\varphi}_0)}-1}{\nu\ds e^{\nu(\norm{u}_0+\norm{\varphi}_0)}}<\frac{1}{\nu}<\tau
$$
as required at the beginning.

Furthermore, $\nu\psi'^2+\psi''=0$ in $\W_a$, so $\Q  w <0$. Besides,  
$$  w (x)=\psi(a)+  \varphi(x)=\norm{u}_0+\norm{\varphi}_0 +  \varphi(x) \geq   u(x) \  \forall \ x\in \partial\W_a\setminus\partial\W$$
because 
$\psi(a) 
=  \norm{u}_0+\norm{\varphi}_0$,
and $w (x)= \psi(0)+\varphi(x)=u(x)$ for $x\in\partial\W$. Since hypothesis \eqref{eq:HY} guaranties that $\Q$ satisfies the maximum principle in $\W$ we conclude that $ u \leq w=\psi\circ \dd + \varphi $ in $\W_a$. 

Anagously we derive $ u \geq - \psi\circ \dd + \varphi $. 
Thus, 
$$ -\psi\circ \dd   \leq u - \varphi \leq \psi\circ \dd  \mbox{ in } \W_a, $$
Consequently,
\begin{equation}\label{est_grad_fronteira_1}
\modulo{\escalar{\nabla u(y)}{N}}\leq \modulo{\escalar{\nabla \varphi(y)}{N}} + \psi'(0).
\end{equation}
Since \eqref{est_grad_fronteira_1} holds and the fact that $u=\varphi$ on $\partial\W$, we conclude
\begin{align*}
\norm{\nabla u(y)}\leq&\norm{\nabla \varphi(y)} + \psi'(0),
\end{align*}
which yield the desire estimate by the definition of $\psi$.
\end{proof}


\begin{cor}\label{teo_Est_gradiente_fronteira_familia}
Let $\W\in M$ be a $\cl^2$ bounded domain and $\varphi\in\cl^2(\overline{\W})$. 
Let 
$H\in\cl^{1}\left({\ocyl\left({\sup\limits_{\overline{\W}}\modulo{u}}\right)}\right)$ 
satisfying 
\begin{equation}
\escalar{\nH}{Y}_{\oM}\geq 0 \mbox{ in } \ocyl,      
\label{eq:HY_familia}
\end{equation}
and
\begin{equation}\label{cond_H_Ricci_sup_est_grad_fronteira_familia}
\norm{\nH\left(\Phi(x,\varphi(y(x)))\right)}_{\oM}\leq \left(H\left(\Phi(x,\varphi(y(x)))\right)\right)^2 + \frac{\oRicc_x}{n} \ \ \forall \ x\in\W_0.
\end{equation} 
Suppose also that 
\begin{equation}\label{cond_Serrin_familia}
\Hc_{\cyl}(y) \geq \ds\sup_{\sigma\in[0,1]}\modulo{H(\Phi(y,\sigma\varphi(y)))} \ \forall \ y\in\partial\W.
\end{equation}
Then, for each $\sigma\in[0,1]$ and each $u\in\cl^2(\W)\cap\cl^1(\overline{\W})$ satisfying 
\begin{equation}\label{ProblemaP_gradiente}
\left\{\begin{array}{l}
 \Q_{\sigma} u =  0\ \mbox{ in }\ \W, \\
 u  =  \sigma \varphi \ \mbox{ in }\  \partial\Omega.
\end{array}\right.
\end{equation}
we have
$$
\sup\limits_{\partial\W}\norm{\nabla u}\leq C,
$$
where $C$ depends on $n$, $\W$, $\min\limits_{\W}\rho$, $\max\limits_{\W}\rho$, $\sup\limits_{\W}\norm{\cYY}$, $\norm{\varphi}_1$, $\norm{\varphi}_2$, 
$$\sup\limits_{\ocyl\left({\sup\limits_{\overline{\W}}\modulo{\varphi}}\right)}\modulo{H},\ 
\sup\limits_{\ocyl\left({\sup\limits_{\overline{\W}}\modulo{\varphi}}\right)}\norm{\nH},\  
\sup\limits_{\W\times\left[-\sup\limits_{\W}\modulo{\varphi},\sup\limits_{\W}\modulo{\varphi}\right]}\norm{D\Phi}.$$
\end{cor}
\begin{proof}
Observe first that \eqref{cond_Serrin_familia} implies
\begin{equation}\label{cond_Serrin_familia_particular_2}
\Hc_{\cyl}(y)\geq \ds\modulo{H(\Phi(y,\sigma\varphi(y)))} \ \forall \ y\in\partial\W, \ \forall\ \sigma\in[0,1]. 
\end{equation} 
For $\sigma=1$ condition \eqref{cond_Serrin_familia_particular_2} becomes
\begin{equation}\label{cond_Serrin_familia_particular}
\Hc_{\cyl}(y)\geq \ds\modulo{H(\Phi(y,\varphi(y)))} \ \forall \ y\in\partial\W.
\end{equation}
As we observed in the proof of the previous theorem, hypothesis \eqref{cond_H_Ricci_sup_est_grad_fronteira_familia} in addition to \eqref{cond_Serrin_familia_particular} imply that \eqref{paraObs} holds. 
From \eqref{paraObs} and \eqref{cond_Serrin_familia_particular_2} we prove exactly as in the proof of Theorem \ref{teo_Est_gradiente_fronteira} that the functions 
$$ \pm \psi\circ \dd + \sigma \varphi $$
are upper and lower barriers for the solution of each Dirichlet problem \eqref{ProblemaP_gradiente} at the boundary. Consequently,
\[
\norm{\nabla u(y)}\leq\norm{\nabla (\sigma\varphi)(y)} + \psi'(0) \leq \norm{\nabla \varphi(y)} + \psi'(0).\qedhere
\]
\end{proof}

\subsection{A priori height estimate}
%


\begin{teo}\label{teo_Est_altura}
Let $\W\in M$ be a $\cl^2$ bounded domain and $\varphi\in\cl^0(\overline{\W})$. 
Let 
$H\in\cl^{1}\left(\overline{\cyl}\right)$ satisfying 
\begin{equation}
\escalar{\nH}{Y}_{\oM}\geq 0 \mbox{ in } \ocyl, 
\label{eq:HY_height}
\end{equation}
and
\begin{equation}\label{cond_H_Ricci_sup}
\norm{\nH\left(\Phi(x,\varphi(y(x)))\right)}_{\oM}\leq \left(H\left(\Phi(x,\varphi(y(x)))\right)\right)^2 + \frac{\oRicc_x}{n} \ \ \forall \ x\in\W_0.
\end{equation} 
Suppose also that 
\begin{equation}\label{cond_Serrin_hightest_teo}
\Hc_{\cyl}(y)\geq \modulo{H(\Phi(y,\varphi(y)))} \ \forall \ y\in\partial\W.
\end{equation}
Then, if $u\in\cl^2(\W)\cap\cl^0(\overline{\W})$ is a solution of \eqref{operador_minimo_1} and $u=\varphi$ in $\partial\W$, we have
$$
\sup\limits_{\partial\W}\norm{\nabla u}\leq C,
$$
where $C$ depends on 
$n$, $\diam(\W)$, $\sup\limits_{\partial\W} \modulo{\varphi}$ and 
$\sup\limits_{\ocyl\left(\sup\limits_{\partial\W}\modulo{\varphi}\right)}\modulo{H}$.
\end{teo}
\begin{proof}
Let 
\begin{equation}
\phi(t)=\dfrac{\ds e^{(\mu/\rho_i)\delta}}{\mu}\left(1-e^{-(\mu/\rho_i) t}\right)
\label{eq:Phi}
\end{equation}
where $\delta=\diam(\W)$, $\rho_i=\inf\limits_{\W}\rho$ and
$\mu$ is a constant to be fixed later on. 
We set $w=\phi\circ \dd  + \sup\limits_{\partial\W}\modulo{\varphi}$,
where $\dd (x)=\dist(x,\partial\W)$ for $x\in\W$. It is known that $\dd \in\cl^2(\W_0)$ (see \cite{LiNirenberg2}).
We will prove that $u\leq w$ in $\overline{\W}$ and the same argument is true for $-u$. 
Let us suppose, by contradiction, that $u-w$ attains a maximum $m>0$ at $x_0\in{\W}$. 

Following the classical argument of Serrin \cite[T. p. 481]{Serrin} for the Euclidean case it can be shown that $x_0$ must be inside $\W_0$. This is exactly the same argument followed on the proof of Theorem 8 on our work \cite{alvarez2019existence} since we are working on the base $M$ (see also \cite{Dajczer2008,spruck}).

On the other hand, from \eqref{eq_transformacao_cte} we have for every $x$ in $\W_0$ that
$$
\Q w = -n\Hc_{\cyl_{\dd}} \phi'\left(\rho^{-2}+\phi'^2\right) + {\rho^{-2}\phi''} -\rho^{-2}\phi' \kappa(x) - n H\left(\Phi(x,w)\right)\left(\rho^{-2}+\phi'^2\right)^{3/2},
$$
where $\kappa(x)$ is the geodesic curvature of $\left\{\Phi_x(t),\ t\in\R\right\}$ at $x$ computed with respect to the normal $\nabla \dd (x)$ and 
$\cyl_{\dd (x)}$ is an hypersurface parallel to a part of $\cyl$ and containing $x$. 

Again, hypothesis \eqref{eq:HY_height} yields $\Dz \left(H\circ \Phi\right)\geq 0$ and 
$w \geq \sup\limits_{\partial\W}\modulo{\varphi}\geq \varphi(y(x))$ by the definition of $\phi$. Thus, 
$$H\left(\Phi(x,w)\right)
\geq H\left(\Phi\left(x,\varphi(y(x))\right)\right). $$
Therefore, 
\begin{equation}\label{Mw_est_altura_0}
\Q w = -n\Hc_{\cyl_{\dd}} \phi'\left(\rho^{-2}+\phi'^2\right) + {\rho^{-2}\phi''} -\rho^{-2}\phi' \kappa(x) - n H\left(\Phi\left(x,\varphi(y(x))\right)\right)\left(\rho^{-2}+\phi'^2\right)^{3/2},
\end{equation}

As in the proof of Theorem \ref{teo_Est_gradiente_fronteira} we observe that the function 
$$h(t)=H\left(\Phi\left(\gamma_y(t),\varphi(y)\right)\right)$$
satisfies the hypothesis of Lemma 6 in \cite{alvarez2019existence}. 
As a consequence, 
\begin{equation}
\modulo{H\left(\Phi(x,\varphi(y(x)))\right)}\leq  \Hc_{\cyl_{\dd (x)}}(x).
\label{eq:inherits}
\end{equation}
It means that property \eqref{cond_Serrin_hightest_teo} is ``inherited'' 
by every hypersurface lying in the interior of the cylinder $\cyl$ and parallel to some part of it. 
This is the key of the proof as we will see in the sequel. 

Using \eqref{eq:inherits} in \eqref{Mw_est_altura_0} it follows
$$
\Q w \leq -n\modulo{H\left(\Phi(x,\varphi(y(x)))\right)}\phi'\left(\rho^{-2}+\phi'^2\right)   + {\rho^{-2}\phi''} -\rho^{-2}\kappa(x)\phi'  - n H\left(\Phi\left(x,\varphi(y(x))\right)\right)\left(\rho^{-2}+\phi'^2\right)^{3/2},
$$

From the definition of $\phi$, \eqref{eq:Phi}, one has
$$
\phi''(t)=-\frac{\mu}{\rho_i} \phi'(t),
$$
then
\begin{align*}
\Q w 
\leq & -n\modulo{H\left(\Phi(x,\varphi(y(x)))\right)}\phi'\left(\rho^{-2}+\phi'^2\right)     
- {\mu}{\rho_i^{-1}} {\rho^{-2} \phi'} 
-\rho^{-2}\kappa(x) \phi' - n H\left(\Phi\left(x,\varphi(y(x))\right)\right)\left(\rho^{-2}+\phi'^2\right)^{3/2},\\
= & -n \modulo{H\left(\Phi(x,\varphi(y(x)))\right)} \phi'\left(\rho^{-2}+\phi'^2\right)  
- \left( {\mu}{\rho_i^{-1}} + \kappa(x)\right) {\rho^{-2}}\phi'
- n H\left(\Phi\left(x,\varphi(y(x))\right)\right)\left(\rho^{-2}+\phi'^2\right)^{3/2}.
\end{align*}
We choose $\mu$ large enough in order to have 
$$  {\mu}{\rho_i^{-1}} + \kappa(x) \geq n \ds\sup_{\ocyl\left(\sup_{\partial\W}\modulo{\varphi}\right)}\modulo{H}.$$
Therefore,
\begin{align*}
\Q w 
\leq & -n\modulo{H\left(\Phi(x,\varphi(y(x)))\right)} \phi'\left(\rho^{-2}+\phi'^2\right)   
- n\modulo{H\left(\Phi(x,\varphi(y(x)))\right)}{\rho^{-2}}\phi'   
- n H\left(\Phi\left(x,\varphi(y(x))\right)\right)\left(\rho^{-2}+\phi'^2\right)^{3/2}\\
= & - n\modulo{H\left(\Phi(x,\varphi(y(x)))\right)}\phi'\left(2\rho^{-2}+\phi'^2\right)
- n H\left(\Phi\left(x,\varphi(y(x))\right)\right)\left(\rho^{-2}+\phi'^2\right)^{3/2}.
\end{align*}
Also, 
$$ \phi'\left(2\rho^{-2}+\phi'^2\right)>\left(\rho^{-2}+\phi'^2\right)^{3/2}$$
since $\phi'\geq\rho.$
Thus
\begin{equation}\label{paraCorolario}
\Q w 
\leq 
 - n \left(\modulo{H\left(\Phi(x,\varphi(y(x)))\right)}
+ H\left(\Phi\left(x,\varphi(y(x))\right)\right)\right)\left(\rho^{-2}+\phi'^2\right)^{3/2}. 
\end{equation}
Hence, $\Q w\leq 0=\Q u$ in $\W_0$. Using \eqref{eq:HY_height}  again we conclude
\begin{align*}
\Q(w+m)
\leq \Q w 
\leq \Q u.
\end{align*}

In addition, $u\leq w + m $ in $\W$ and $u(x_0) = w(x_0) + m $ because $u-w\leq u(x_0)-w(x_0)=m$. 
Since the maximum principle holds for the operator $\Q$ we conclude that $u = w+m$ in $\W_0$ which is impossible since $u<w+m$ in $\partial\W$. Consequently $u\leq w$ in $\overline{\W}$. 
\end{proof}

\begin{cor}\label{teo_Est_altura_familia}
Let $\W\in M$ be a $\cl^2$ bounded domain and $\varphi\in\cl^0(\overline{\W})$. 
Let 
$H\in\cl^{1}\left(\overline{\cyl}\right)$ satisfying 
\begin{equation}
\escalar{\nH}{Y}_{\oM}\geq 0 \mbox{ in } \ocyl, 
\label{eq:HY_height_familia}
\end{equation}
and
\begin{equation}\label{cond_H_Ricci_sup_familia}
\norm{\nH\left(\Phi(x,\varphi(y(x)))\right)}_{\oM}\leq \left(H\left(\Phi(x,\varphi(y(x)))\right)\right)^2 + \frac{\oRicc_x}{n} \ \ \forall \ x\in\W_0.
\end{equation} 
Suppose also that 
\begin{equation}\label{cond_Serrin_hightest_teo_familia}
\Hc_{\cyl}(y)\geq \modulo{H(\Phi(y,\varphi(y)))} \ \forall \ y\in\partial\W.
\end{equation}
Then, for each $\sigma\in[0,1]$ and each $u\in\cl^2(\W)\cap\cl^1(\overline{\W})$ satisfying
\begin{equation}\label{ProblemaP_gradiente}
\left\{\begin{array}{l}
\Q_{\sigma} u =  0\ \mbox{ in }\ \W, \\
 u  =  \sigma \varphi \ \mbox{ in }\  \partial\Omega.
\end{array}\right.
\end{equation}
we have
$$
\sup\limits_{\partial\W}\norm{\nabla u}\leq C,
$$
where $C$ depends on 
$n$, $\diam(\W)$, $\sup\limits_{\partial\W} \modulo{\varphi}$ and 
$\sup\limits_{\ocyl\left(\sup\limits_{\partial\W}\modulo{\varphi}\right)}\modulo{H}$.
\end{cor}
\begin{proof}
Let $u$ be a solution of \eqref{ProblemaP_gradiente}. Being $w$ the function defined in the proof of the previous theorem we suppose, by contradiction, that $u-w$ attains a maximum $m>0$ at $x_0\in{\W}$, that is,
\begin{equation}
\left\{
\begin{array}{l}
u\leq w + m   \mbox{ in } \W\\
u(x_0) = w(x_0) + m. 
\end{array}\right.
\label{eq:maximum}
\end{equation} 
We already know that $x_0$ must be in $\W_0$. 

On the other hand, 
exactly as we derived \eqref{paraCorolario}, we obtain
\begin{align*}
\Q_{\sigma} w 
\leq &
- n \left(\modulo{H\left(\Phi(x,\varphi(y(x)))\right)}
+ \sigma H\left(\Phi\left(x,\varphi(y(x))\right)\right)\right)\left(\rho^{-2}+\phi'^2\right)^{3/2}\\
\leq &
- n \sigma \left(\modulo{H\left(\Phi(x,\varphi(y(x)))\right)}
+ H\left(\Phi\left(x,\varphi(y(x))\right)\right)\right)\left(\rho^{-2}+\phi'^2\right)^{3/2}\\
\end{align*}
Therefore, 
\begin{equation}
\Q_{\sigma} (w+m) \leq \Q_{\sigma} w \leq 0 = \Q_{\sigma} u. 
\label{eq:maximum_2}
\end{equation} 

The maximum principle in combination with \eqref{eq:maximum} and \eqref{eq:maximum_2} implies that $u = w+m$ in $\W_0$, being that 
$$ u = \sigma \varphi\leq\sup_{\partial\W}\modulo{\varphi} = w < w+m .$$ 
Therefore $u\leq w$ in $\overline{\W}$. 

\end{proof}

\section{Proof of Theorem \ref{T_exist_Ricci}}

Under the assumptions of Theorem \ref{T_exist_Ricci} we have an a priori global estimate in $\cl^{1}(\overline{\W})$ for the solutions of the related problems 
$$
\left\{\begin{array}{r}
W^2\Delta u - \Hess u(\nabla u,\nabla u)-\rho^{-2}\left(\rho^{-2}+W^2\right)\escalar{\cYY}{\nabla u} = \sigma nH(\Phi(x,u))W^3 \ \mbox{ in }\ \W, \\
      u = \sigma \varphi \ \mbox{ in }\  \partial\Omega, 
\end{array}\right.
$$
(independent on $\sigma$ and $u$) in view of Theorem \ref{teo_Est_global_gradiente} and Corollaries \ref{teo_Est_gradiente_fronteira_familia} and \ref{teo_Est_altura_familia}. The Leray-Schauder fixed point theorem and a theorem of Ladyzhenskaya and Ural`tseva \cite[Th. 13.7]{GT} classically applied guarantee the existence of a solution of the Dirichlet problem
\begin{equation}\label{ProblemaP}
\left\{\begin{array}{r}
W^2\Delta u - \Hess u(\nabla u,\nabla u)-\rho^{-2}\left(\rho^{-2}+W^2\right)\escalar{\cYY}{\nabla u} = nH(\Phi(x,u))W^3 \ \mbox{ in }\ \W, \\
      u = \varphi \ \mbox{ in }\  \partial\Omega.
\end{array}\right.
\end{equation}

Under condition \eqref{eq:HY_intro} the operators $\Q$ and $\Q_{\sigma}$ satisfy the maximum principle. Therefore, there exists a unique solution to \eqref{ProblemaP}.

\bibliographystyle{amsplain}
\bibliography{bibliografia}

%
%

\end{document}